\documentclass[a4paper]{article}
\usepackage[all]{xy}\usepackage[latin1]{inputenc}        
\usepackage[dvips]{graphics,graphicx}
\usepackage{amsfonts,amssymb,amsmath,color,mathrsfs, amstext}
\usepackage{amsbsy, amsopn, amscd, amsxtra, amsthm,authblk}
\usepackage{enumerate,algorithmic,algorithm}
\usepackage{fancyhdr}
\usepackage{ulem}
\usepackage{bbm}
\usepackage{upref,ulem}
\usepackage{geometry}
\usepackage{float}
\usepackage{subfigure}
\usepackage{yhmath}
\usepackage{booktabs}
\usepackage{graphicx}
\usepackage{multirow}
\usepackage{makecell}
\usepackage[colorlinks,
            linkcolor=red,
            anchorcolor=red,
            citecolor=red,
            urlcolor=black,
            ]{hyperref}

\numberwithin{equation}{section}

\newcommand{\R}{\mathbb{R}}

\newcommand{\seq}{\operatorname{seq}}

\def\R{\mathbb{R}}

\newcommand \footnoteONLYtext[1]
{
	\let \mybackup \thefootnote
	\let \thefootnote \relax
	\footnotetext{#1}
	\let \thefootnote \mybackup
	\let \mybackup \imareallyundefinedcommand
}

\definecolor{DarkBlue}{rgb}{0,0,.8}
\newtheorem{theorem}{Theorem}[section]
\newtheorem{lemma}{Lemma}[section]
\newtheorem{definition}{Definition}[section]

\newtheorem{proposition}{Proposition}[section]

\newtheorem{corollary}{Corollary}[section]
\numberwithin{equation}{section}

\providecommand{\keywords}[1]
{
  \small	
  \textbf{\textit{Keywords:}} #1
}

\providecommand{\msc}[1]
{
  \small	
  \textbf{\textit{Mathematics Subject Classification:}} #1
}

\begin{document}

\title{Convergence analysis of OT-Flow for sample generation}

\author[1]{Yang Jing \thanks{sharkjingyang@sjtu.edu.cn} }
\author[2]{Lei Li \thanks{leili2010@sjtu.edu.cn} }
\affil[1,2]{School of Mathematical Sciences, Shanghai Jiao Tong University, Shanghai, 200240, P.R.China.}
\affil[2]{Institute of Natural Sciences, MOE-LSC, Shanghai Jiao Tong University, Shanghai, 200240, P.R.China.}
\affil[2]{Shanghai Artificial Intelligence Laboratory}

\date{}
\maketitle

\begin{abstract}
Deep generative models aim to learn the underlying distribution of data and generate new ones. Despite the diversity of generative models and their high-quality generation performance in practice, most of them lack rigorous theoretical convergence proofs. In this work, we aim to establish some convergence results for OT-Flow, one of the deep generative models. First, by reformulating the framework of OT-Flow model, we establish the $\Gamma$-convergence of the formulation of OT-flow to the corresponding optimal transport (OT) problem as the regularization term parameter $\alpha$ goes to infinity. Second, since the loss function will be approximated by Monte Carlo method in training, we established the convergence between the discrete loss function and the continuous one when the sample number $N$ goes to infinity as well. Meanwhile, the approximation capability of the neural network provides an upper bound for the discrete loss function of the minimizers. The proofs in both aspects provide convincing assurances for OT-Flow.

\end{abstract}
\\
\msc{49Q22; 68T07}

\section{Introduction}\label{sec:intro}
Deep generative models \cite{goodfellow2014generative,DBLP:journals/corr/KingmaW13, kingma2019introduction,rezende2015variational} exhibit promising performance across a variety of tasks, including image generation\cite{dhariwal2021diffusion}, text-to-image generation\cite{ramesh2022hierarchical} and video generation\cite{liu2024sora}. The widely-used frameworks include diffusion probabilistic models (DPMs)\cite{ho2020denoising,song2020denoising,bao2022analytic}, continuous normalizing flows (CNFs)\cite{chen2018neural,grathwohl2018ffjord}, variational auto-encoders (VAEs)\cite{DBLP:journals/corr/KingmaW13, kingma2019introduction} and generative adversarial networks (GANs)\cite{goodfellow2014generative,arjovsky2017wasserstein}. Among above four popular frameworks, CNFs are characterized by continuous-time ordinary differential equations (ODEs), and DPMs utilize stochastic differential equations (SDEs) as their backbone. Through DPMs and CNFs, samples evolve from data points to Gaussian distribution in the forward process and gradually remove noise  to generate samples in the backward process. In comparison with GANs and VAEs, samples of DPMs and CNFs are generated in smoother way, not only achieving superior sample quality but also enabling exact likelihood computation. Despite the diversity of generative models and their outstanding performance in downstream tasks, the mathematical principles behind the models and rigorous convergence proofs are developed far behind the rapid iteration of the models. In this paper, our focus lies in establishing convergence results for OT-Flow, which stands as one of the most practical CNFs. Such convergence analysis ensures stability during the training and aids in comprehending the underlying mechanisms of the model.

The continuous normalizing flows (CNFs) are a class of sample generative models based on particle transportation purely. The CNFs aim to build continuous and invertible mappings between an arbitrary distribution $\rho_{0}$ and standard normal distribution $\rho_{1}$ by setting the velocity field as an output of neural network. In particular, for a given time $T$, one is trying to obtain a mapping $z: \R^{d} \times [0,T]\rightarrow \R^{d}$, which defines a continuous evolution $x \mapsto z(x, t) $ of every $x \in \R^{d}$. Then the density $\rho(z(x,t),t)$ satisfies
\begin{equation}\label{original density formulation}
    \log \rho_{0}(x)=\log \rho(z(x,t),t)+\log|\det \nabla z(x,t)| \quad \text{for all} \quad x \in \R^{d}.
\end{equation}
Especially at time $T$ we have $\log \rho_{0}(x)=\log \rho_{1}(z(x,T),T)+\log|\det \nabla z(x,T)| $. Define $\ell(x, t):=\log |\det \nabla z(x, t)|$, then $z(x,t)$ and $\ell(x,t)$ satisfy the following ODE system
\begin{equation}\label{eq:original ODE}
\partial_{t}\left[\begin{array}{c}
z(x, t) \\
\ell(x, t)
\end{array}\right]=\left[\begin{array}{c}
v(z(x, t), t ; \boldsymbol{\theta}) \\
\operatorname{tr}(\nabla v(z(x, t), t ; \boldsymbol{\theta}))
\end{array}\right], \quad\left[\begin{array}{c}
z(x, 0) \\
\ell(x, 0)
\end{array}\right]=\left[\begin{array}{c}
x \\
0
\end{array}\right].
\end{equation}
To train the dynamics, CNFs  minimize the expected negative log-likelihood given by the right-hand-side in \eqref{original density formulation}, or equivalently the KL divergence between target distribution and final distribution under the constraint \eqref{eq:original ODE} \cite{rezende2015variational,papamakarios2017masked,papamakarios2021normalizing,grathwohl2018ffjord}:
\[
J=\mathbb{KL}\left[\rho(\boldsymbol{z}(\boldsymbol{x}, T)) \| \rho_1(\boldsymbol{z}(\boldsymbol{x}, T))\right].
\]
For convenience we solve \eqref{eq:original ODE} together to obtain the change of $\rho$, which will lead to a more efficient estimation of density. 

From the ODE system \eqref{eq:original ODE}, we can see that once the velocity field is learned, one can track the evolution of density and invert the transport map by running the ODE backward. The invertibility of CNFs grants us access to estimate the density of the sample space, which can be employed for density estimation and Bayesian inference.

In general, the velocity field that transforms a given probability measure to a target one is not unique in the formulation of CNFs. One should observe the parallels between CNFs and the dynamic formulation of Wasserstein distance, where both involve the evolution of probability distributions through mass transportation under a velocity field. It is natural to incorporate optimal transport concepts into CNFs to enhance algorithm performance. Finlay et al. \cite{finlay2020train} pioneered the introduction of optimal transportation regularization in normalizing flows, while Onken et al. \cite{onken2021ot} subsequently proposed OT-Flow as an enhanced version of CNFs, which leverages optimal transport theory to regularize the CNFs and enforce straight trajectories that are easier for numerical integration. OT-Flow may be preferred in applications due to the particles' straight-line paths and trajectories avoiding intersections. Consequently, such a model is expected to improve its invertibility and generation efficiency.

The optimal transport (OT) was first introduced by Monge \cite{monge1781memoire} in 1781 and was relaxed by Kantorovich \cite{kantorovich2006problem}. The optimal transport theory actually provides a specific way to transform the measure $\mu$ to $\nu$ with minimum transportation cost. In particular, let $\Omega\subset \R^d$. Given two distribution $\mu$ and $\nu\in \mathcal{P}(\Omega)$, where $\mathcal{P}(\Omega)$ is the set of all probability measures on $\Omega$. One can define the Wasserstein-$p$ distance ($p \geq 1$) between $\mu$ and $\nu$ :
\begin{equation*}
W_{p}(\mu, \nu)=\left(\inf _{\gamma \in \Pi(\mu, \nu)} \int|x-y|^{p} d \gamma\right)^{1 / p},
\end{equation*}
where $\Pi(\mu, \nu)$ is the set of {\it transport plans}, i.e. a joint measure on $X \times Y$, with marginal distribution $\mu$ and $\nu$. Define the space 
\[
\mathcal{W}_{p}:=\{\mu \in \mathcal{P}(\Omega): \int_\Omega |x|^{p} \mu(dx)< \infty\}.
\]
Then $(\mathcal{W}_{p},W_{p})$ is a complete metric space. Wasserstein distance stands out as a prominent choice among metrics due to its ability to quantify dissimilarity between two distributions, even in cases where one or both distributions consist of discrete data samples with disjoint supports, which can be applied to traditional models for improvements \cite{onken2021ot,DBLP:conf/iclr/SalimansZRM18}. On a convex and compact domain $\Omega$, the Wasserstein distance $W_{p}$ admits the following dynamic Benamou-Brenier formulation \cite[Chap. 5 Theorem 5.28]{santambrogio2015optimal}.  Let $\mu$ and $\nu$ are two probability distribution on $\Omega$ and are absolutely continuous with respect to the Lebesgue measure, and $v_t$ is a vector field on $\Omega$,
\begin{equation}\label{eq:dynamic_ot}
W_{p}^{p}(\mu, \nu)=\min_{\rho, v}\left\{\int_{0}^{1}\left\|v_{t}\right\|_{L^{p}(\rho)}^{p} d t: \partial_{t} \rho_{t}+\nabla \cdot\left(\rho_{t} v_{t}\right)=0, v\cdot n|_{\partial\Omega}=0, \rho_{0}=\mu, \rho_{1}=\nu\right\},
\end{equation}
where $\left\|v_{t}\right\|_{L^{p}(\rho_{t})}^{p}=\int_{\Omega}\left|v_{t}(x)\right|^{p} \rho_{t}(dx)$. 
Clearly, the optimal velocity field $v_t$ in this problem is one of the candidates for CNFs and could be optimal in certain sense. The cost in the OT-Flow to train the velocity field is then given as follows
\begin{equation}\label{eq:Cost1}
    J=\mathbb{KL}\left[\rho(x, T) \| \rho_{1}(x)\right]+\frac{2}{\alpha}\mathbb{E}_{\rho_{0}} \left[\int_{0}^{T} \frac{1}{2}|v(z(x, t), t)|^{2} d t \right].
\end{equation}
Here $\alpha$ is a hyperparameter to balance KL divergence and trajectory penalty. Note that even when $\Omega$ is not compact or not convex when there is some possibility that dynamic formulation is not strictly equal to the Wasserstein distance, such a formulation could still be beneficial since the dynamical formulation itself corresponds to some metric.

The KL divergence term in \eqref{eq:Cost1} serves as a soft terminal constraint, which enforces the terminal distribution $\rho(x,T)$ transported by velocity field to get close to $\rho_1$. The second term is related to the dynamic Benamou-Brenier formulation of $W_{2}$ distance in optimal transport theory, which can also be regarded as a penalty of the squared arc-length of the trajectories. Ideally if the KL divergence term is zero, minimizing the cost function is equivalent to minimizing $W_{2}$ distance and solving the optimal velocity field, which will  encourage straight trajectory.

OT-Flow proves practical in sample generation, gaining an advantage over previous CNFs by integrating optimal transport (OT) concepts. One may find it interesting to explore the relationship between the velocity field solved by neural networks in OT-Flow and the classical solutions of OT. In our work, we conduct a convergence analysis for OT-Flow. More specifically, our convergence analysis  mainly contain two parts. In the first part, we reformulate OT-Flow and classical OT problems into continuous optimization problems with similar form. Subsequently, we demonstrate that OT-Flow  $\Gamma$-converges to OT as the regularization coefficient $\alpha \to \infty$, indicating the minimizers of OT-Flow will converges to ones of OT. One should notice that OT-Flow uses data samples to approximate the equivalent loss functional. In the second part, we illustrate that as with an increasing dataset size, the minimizers of whose loss functional is approximated through Monte Carlo method, will eventually converge to theoretical solutions with a sufficient neural network approximation capability.

The rest of the paper is organized as follows. Section \ref{sec:setup and notations} provides some setup and notations to formulate the problem. We provide an overview of $\Gamma$-convergence as well, which serves as a fundamental tool in our analysis. Then we establish a convergence analysis between OT-Flow and classical OT problems in section \ref{Convergence from OT-Flow to OT}. In section \ref{sec: convergence of N}, we consider the convergence of the minimizers with respect to sample number $N \to \infty$. We conclude the work and make a discussion in section \ref{sec:dis}.

\section{Setup and notations}\label{sec:setup and notations}

Recall that the optimization problem of OT-Flow in \eqref{eq:Cost1}. Assume $D \subset \mathbb{R}^d$ is a bounded domain with smooth boundary.  The KL divergence (relative entropy) between two probability measure $\mu$ and $\nu$ on $\mathbb{R}^d$ is defined by
\begin{equation*}
    \mathbb{KL}\left[\mu|| \nu\right]=\left\{\begin{array}{lr}
    \int_{D} \log (\frac{d \mu}{d \nu}) d\mu, & \text { if } \mu \ll \nu \\
    \ \infty, & \text { else, }
    \end{array}\right.
\end{equation*}
where $\frac{d\mu}{d \nu}$ denotes the Radon-Nikodym derivative of $\mu$ with respect to $\nu$. Note that the KL divergence is non-negative by Jensen's inequality and achieves zero only if $\mu= \nu$. Moreover, it is a convex functional with respect to either argument. To obtain the full description of the mathematical problems, one also needs to specify the no-flux boundary condition $v \cdot n =0$ on $\partial D$ by the physical significance. Hence, the optimization problem becomes
\begin{equation}\label{minimization problem:OT-flow v rho}
\begin{array}{ll}
\min\limits_{\rho, v} & \mathbb{KL}\left[\rho(x,T)||\rho_1(x)\right]+ \frac{1}{\alpha}\int_{0}^{T} \int_{D}\rho |v|^2 dx dt, \\
\text { s.t. } &\partial_{t} \rho +\nabla \cdot(\rho v)=0  \text { in } D \times [0,T] , \\
\text{with} & \rho(x,0)=\rho_{0}(x), \quad v(x,t)\cdot n=0
\text{ on }  \partial D \times [0,T].
\end{array}
\end{equation}
However, the minimization problem above in the variables $(\rho,v)$ has nonlinear constraints, it is found convenient to switch variables from $(\rho, v)$ into $(\rho, m)$ where $m=\rho v$. Thus the full description of the optimization problem of OT-Flow gives as follows:
\begin{equation}
\label{minimization problem:OT-flow m rho  1/alpha}
\begin{array}{ll}
\min\limits_{\rho, m} & \mathbb{KL}\left[\rho(x,T)||\rho_1(x)\right]+\frac{1}{\alpha}\int_{0}^{T} \int_{D}\frac{|m|^2}{\rho} dx dt, \\
\text { s.t. } &\partial_{t} \rho +\nabla \cdot m=0\text { in } D \times [0,T] , \\
\text{with} & \rho(x,0)=\rho_{0}(x),\\
& m(x,t) \cdot n=0  \text{ on }  \partial D \times [0,T].
\end{array}
\end{equation}
Correspondingly, the optimal transport problem then becomes:
\begin{equation}
\label{minimization problem:OT m rho}
\begin{array}{ll}
\min\limits_{\rho, m} & \int_{0}^{T} \int_{D}\frac{|m|^2}{\rho} dx dt, \\
\text { s.t. } &\partial_{t} \rho +\nabla \cdot m=0\text { in } D \times [0,T] , \\
\text{with} & \rho(x,0)=\rho_{0}(x),\rho(x,T)=\rho_{1}(x),\\
& m(x,t) \cdot n=0  \text{ on }  \partial D \times [0,T].
\end{array}
\end{equation}

In section \ref{Convergence from OT-Flow to OT}, we will first study the convergence of \eqref{minimization problem:OT-flow m rho 1/alpha} to \eqref{minimization problem:OT m rho}. We will allow $m$ to be measures and the exact meaning of $|m|^2/\rho$ will be explained in section \ref{sec:Lower semi-continuity}. Moreover, the working space we choose in this work is
\begin{equation*}
\rho \in L^{1}([0,T];\mathcal{W}_{2}(D)), \quad
m \in L^{1}([0,T];\mathscr{M}(D)).
\end{equation*}
Then, the problems \eqref{minimization problem:OT-flow m rho 1/alpha}--\eqref{minimization problem:OT m rho} become convex optimization and the corresponding functionals have the lower semicontinuity with respect to the weak convergence (see section \ref{sec:Lower semi-continuity} for details).

The main tool we use is the $\Gamma$-convergence \cite{braides2002gamma}. Here, let us briefly introduce the relevant concepts.
We will collect some basics of  $\Gamma$-convergence, which is commonly used to investigate the convergence of the optimization problems and their optimizers. Here we first recall the definitions of the usual $\Gamma$-convergence:
\begin{definition}\label{def:Gmconv}
Let $X$ be a topological space. Let $(f_n)$ be a sequence of functionals on $X$. Define
    \begin{gather*}
    \begin{split}
    & \Gamma\hbox{-}\limsup_{n\to\infty} f_n(x)=\sup_{N_x}\limsup_{n\to\infty}\inf_{y\in N_x}f_n(y),\\
    & \Gamma\hbox{-}\liminf_{n\to\infty} f_n(x)=\sup_{N_x}\liminf_{n\to\infty}\inf_{y\in N_x}f_n(y),
    \end{split}
    \end{gather*}
    where $N_x$ ranges over all the neighborhoods of $x$. If there exists a functional $f$ defined on $X$ such that
    \begin{gather*}
    \Gamma\hbox{-}\limsup_{n\to\infty} f_n=\Gamma\hbox{-}\liminf_{n\to\infty} f_n =f,
    \end{gather*}
    then we say the sequence $(f_n)$ $\Gamma$-converges to $f$.
\end{definition}
\begin{proposition}
    Any cluster point of the minimizers of a $\Gamma$-convergent sequence $(f_n)$ is a minimizer of the corresponding $\Gamma$-limit functional $f$.
\end{proposition}
Thus $\Gamma$-convergence is an ideal tool to describe the convergence of the minimizers of a sequence of optimization problems. Especially, one may apply $\Gamma$-convergence to study the asymptotic behavior of neural network solutions as regularization coefficient goes to infinity or zero.

Direct verification of $\Gamma$-convergence using Definition \ref{def:Gmconv} can be challenging in most cases. Some refined versions of the $\Gamma$-convergence were proposed in \cite{buttazzo1982gamma}, which are known as  $\Gamma_{\seq}$-convergence. There are many notions of $\Gamma_{\seq}$-convergence.  Here, we will introduce two of them, which will be used to establish  $\Gamma$-convergence analysis from OT-Flow to OT in section \ref{Convergence from OT-Flow to OT}. 
In particular, we will consider
\begin{equation*}
\Gamma_{\seq}(\mathbb{N}^+, X^-)\lim_n f_n := \inf_{x^n \rightarrow x} \limsup_{n\rightarrow \infty} f_n(x_n),   
\end{equation*}
and
\begin{equation*}
\Gamma_{\seq}(\mathbb{N}^-, X^-)\lim_n f_n  := \inf_{x^n \rightarrow x} \liminf_{n\rightarrow \infty} f_n(x_n).
\end{equation*}
Here, the $\inf_{x^n \rightarrow x}$ means the infimum is taken with respect to all sequences $\{x_n\}$ that converge to $x$.
The following proposition gives the closed relationship between $\Gamma_{\seq}$-convergence and the usual $\Gamma$-convergence. We will omit the proof and a more rigorous proof can be found in \cite{xiong2024convergence}.
\begin{proposition}\label{connection}
Suppose that $X$ is a first-countable topological space. It holds that
    \begin{gather*}
        \begin{aligned}
            &\inf_{x^n \rightarrow x} \limsup_{n\rightarrow \infty} f_n(x_n) =\Gamma\hbox{-}\limsup_{n\to\infty} f_n ,\\
            &\inf_{x^n \rightarrow x} \liminf_{n\rightarrow \infty} f_n(x_n) =\Gamma\hbox{-}\liminf_{n\to\infty} f_n .
        \end{aligned}
    \end{gather*}
    Consequently, if 
    \begin{equation*}
        \inf_{x^n \rightarrow x} \limsup_{n\rightarrow \infty} f_n(x_n)=\inf_{x^n \rightarrow x} \liminf_{n\rightarrow \infty} f_n(x_n) := f
    \end{equation*} exists, then $(f_n)$ $\Gamma$-converges to $f$.
\end{proposition}

As mentioned in the introduction, OT-Flow utilizes a neural network to parameterize the velocity field and optimize the cost functional in \eqref{eq:Cost1}. However, training the model directly with the cost function is not feasible as we only have discrete samples from $\rho_{0}$. To address this, one can approximate the cost function using Monte Carlo methods, which needs to rewrite cost function in the form of expectation over $\rho_{0}$. According to \cite{onken2021ot}, one may simplify the KL divergence term with density relationship \eqref{original density formulation} and drop constant in the formulation. The final cost function $J$ of OT-Flow gives as follows:
\begin{equation}\label{eq:OT-Flow}
\begin{split}
    J&=\mathbb{E}_{\rho_{0}(x)} \left[ C(x,T)+\frac{2}{\alpha}L(x,T) \right],\\
    C(x,T)&=-\ell(x,T) +\frac{1}{2}|\boldsymbol{z}(\boldsymbol{x}, T)|^{2}+\frac{d}{2} \log (2 \pi),\\
    L(x,T)&=\int_{0}^{T} \frac{1}{2}|v(z(x, t), t)|^{2} d t.
\end{split}
\end{equation}
In particular, from the Pontryagin Maximum Principle \cite{evans1983introduction}, there exists a potential function $\Phi: \mathbb{R}^d \times [0,T] \rightarrow \mathbb{R}$ such that
\begin{equation*}
    v(x,t)=-\nabla \Phi(x,t)
\end{equation*}
Thus OT-Flow parameterize $\Phi$ with a neural network instead of $v$ in practice. We will concentrate on using neural networks to directly approximate the velocity field for the sake of notation convenience, which will have no impact on our analysis.
One should remember that the discretion of cost function will introduce variations to the minimizers as well. Hence we will then show that as sample number $N\to\infty$, the optimal solutions of neural network will converges to the theoretical minimizers of the cost functional in section \ref{sec: convergence of N}.

\section{Convergence from OT-Flow problem to optimal transport problem}\label{Convergence from OT-Flow to OT}

In this section, we will consider the OT flow and the optimal transport problems for given $\rho_0\in \mathcal{P}(D)$ and $\rho_1\in \mathcal{P}(D)$.  For the convenience of proof, we can rewrite the OT problem \eqref{minimization problem:OT m rho} in the following form to keep the constraints consistent:
\begin{equation}
\label{minimization problem:OT indicator}
\begin{array}{ll}
\min\limits_{\rho, m} & \int_{0}^{T} \int_{D}\frac{|m|^2}{\rho} dx dt +\mathbf{1}_E , \\
\text { s.t. } &\partial_{t} \rho +\nabla \cdot m=0\text { in } D \times [0,T] , \\
\text{with} & \rho(x,0)=\rho_{0}(x),\\
& m(x,t) \cdot n=0  \text{ on }  \partial D \times [0,T].
\end{array}
\end{equation}
where $E$ is the set of the terminal constraints
\begin{equation*}
    E:= \left\{ \rho \in X :  \rho (x,T)  =   \rho_1(x)  \right\},
\end{equation*}
and $\mathbf{1}_E(x)$ the indicator function for a set $E$ in $X$:
\begin{equation*}
\mathbf{1}_E(x)=\left\{\begin{array}{lr}
0, & \text { if } x \in E, \\
+\infty, & \text { otherwise. }
\end{array}\right.
\end{equation*}
The OT-Flow problem \eqref{minimization problem:OT-flow m rho  1/alpha} can be equivalently written as
\begin{equation}
\label{minimization problem:OT-flow m rho alpha}
\begin{array}{ll}
\min\limits_{\rho, m} & \int_{0}^{T} \int_{D}\frac{|m|^2}{\rho} dx dt+\alpha \mathbb{KL}\left[\rho(x,T)||\rho_1(x)\right], \\
\text { s.t. } &\partial_{t} \rho +\nabla \cdot m=0\text { in } D \times [0,T] , \\
\text{with} & \rho(x,0)=\rho_{0}(x),\\
& m(x,t) \cdot n=0  \text{ on }  \partial D \times [0,T].
\end{array}
\end{equation}
The reason to choose this form because it is more close to the optimal transport problem \eqref{minimization problem:OT indicator}.

Our goal is to prove that the minimization problem \eqref{minimization problem:OT-flow m rho alpha} is $\Gamma-$convergent to the minimization problem \eqref{minimization problem:OT indicator} when $\alpha \to \infty$. Next, we need to make the framework
of the problems rigorous, like the topology of the space and the significance of the constraint etc.

\subsection{The topological properties of the working spaces}

We will focus on the properties of the space we will work on: 
\begin{equation*}
\begin{split}
& X = L^{1}([0,T];\mathcal{W}_{2}(D))
    :=\left\{\rho: [0, T]\to \mathcal{W}_{2}(D)) \mid \int_0^T \int_{D}|x|^2\rho_t(dx)dt<\infty \right\},\\
& \quad Y = L^{1}([0,T];\mathscr{M}^d(D))
:=\left\{m: [0, T]\to \mathscr{M}^d(D)) \mid
\int_0^T\|m\|_{D}dt<\infty \right\}. 
\end{split}
\end{equation*}
Here, $\mathscr{M}(D)$ is the set of signed measures on $D$
and $\mathscr{M}^d(D)$ is the set of vector-valued signed measures on $D$ with $d$ components. The notation $\|m\|_{D}$ indicates the total variation norm which we explain below. 

Let $C_b$ be the set of all continuous functions on $D$ (the subindex "$b$" indicates that it is bounded, which is natural since continuous functions on bounded set are bounded). Let $C_b^1$ be the set of all continuously differentiable functions on $D$ and the first order derivatives are bounded on $D$. Without explicitly stated, $C_b$ is equipped with the uniform convergence norm. Under the uniform convergence norm
\[
\|\varphi\|:=\|\varphi\|_{\infty}=\sup_{x\in D}|\varphi(x)|,
\]
$C_b^1$ is a dense subset of $C_b$ (of course, if one considers the norm $\|\varphi\|_{\infty}+\|D\varphi\|_{\infty}$, $C_b^1$ itself is complete). 

For a measure $\mu \in \mathscr{M}(D)$, the total variation norm is the dual operator norm over $C_b$:
\begin{gather*}
\|\mu\|_{D}:=\sup_{\varphi\in C_b, \|\varphi\|\le 1}\int_{D}\varphi \mu(dx).
\end{gather*}
Similarly, if $f$ is a vector-valued measure, the total variation norm is defined by
\begin{gather*}
\|f\|_{D}:=\sup\left\{\int_{D}g\cdot f(dx):  g: D\to \R^d 
~\text{continuous}, \sup_{x}|g(x)|\le 1 \right\}.
\end{gather*}
Here, $|g(x)|$ is the Euclidean length.
If $\mu$ and $f$ are absolutely continuous to some nonnegative measure $\lambda$ (for example, the Lebesgue measure $dx$), then the total variation norms for the scalar measure $\mu$ and vector valued measure $f$ on $D$ are as follows: 
\begin{equation*}
\|\mu\|_{D}=\int_{D} |\mu(x)| \ d\lambda , \ \ \| f \|_{D}=\int_{D} |f(x)|_2 \ d\lambda,
\end{equation*}
where we abuse the notations for convenience and understand $\mu(x)=d\mu/d\lambda$ and $f(x)=df/d\lambda$ as the densities with respect to $\lambda$. The notation $|f(x)|_2$ is the Euclidean length of the vector $f(x)$.

In our case, when $m$ and $\rho$ are curves in $\mathcal{W}_2$ and $\mathscr{M}^d(D)$, the total variation norms could be taken for each $t\in [0, T]$. In our case $\rho$ is a probability measure, so $\|\rho_t\|_{D}=1$ for each $t$, and thus $\int |x|^2 \rho(dx)<\infty$ since $D$ is bounded. Then, for every measurable curve taking values in $\mathcal{P}(D)$, one then has 
\begin{gather*}
\|\rho\|=\int_0^T \|\rho(t)\|_{D} dt=T, \quad \int_0^T \int_{D} |x|^2 \rho(dx)\,dt<\infty.
\end{gather*}
The norm for $m \in Y$ is then
\begin{equation*}
\|m\|=\int_0^T \|m(t)\|_{D} dt.
\end{equation*}

Similar as in \cite{xiong2024convergence}, we equip $X$ and $Y$ with the weak topology respectively: we say $\rho_n \Rightarrow \rho$ in $X$ if
\begin{equation*}
    \int_0^T \int_{D} \varphi d\rho_{n} \rightarrow \int_0^T \int_{D} \varphi d\rho, \forall \varphi \in C_{b}(\bar{D} \times [0,T];\mathbb{R}),
\end{equation*}
and $m_n \Rightarrow m$ in $Y$ if 
\begin{equation*}
    \int_0^T \int_{D} g \cdot dm_{n} \rightarrow \int_0^T \int_{D} g \cdot dm, \forall  g \in C_{b}(\bar{D} \times [0,T];\mathbb{R}^d) .
\end{equation*}
(Note that there is a typo in Eqn. (3.18)  Page 758 of \cite{xiong2024convergence} where $C_b^1$ should be $C_b$.) The weak topology used here as the dual of the class of $C_b$ functions is standard in literature. If both $m_n$ and $\rho_n$ uniformly bounded in total variation, then one can replace the test functions from $C_b$ to $C_b^1$ in \eqref{def:topology of product space} since the set $C_b^1$ is dense in $C_b$ under the topology of uniform convergence.

With the weak topology introduced, we have the following fact about the space.
\begin{proposition}
Both $X$ and $Y$ are closed subspaces of the signed measures and vector-valued signed measures with respect to the weak topology in the sense that
\begin{enumerate}[(i)]
\item If $\rho^{(n)}\in X$ and there is some $\rho: [0, T]\to \mathscr{M}(D)$ such that $\rho^{(n)} \Rightarrow \rho$, then $\rho\in X$. 
\item If $m^{(n)}\in Y$ and there is some $m: [0, T]\to \mathscr{M}^d(D)$ such that $m^{(n)} \Rightarrow m$, then $m\in Y$.
\end{enumerate}

\end{proposition}
\begin{proof}

Consider a sequence $\{\rho^{(n)}\} \in X$ such that $\rho_n \Rightarrow \rho$. First, since $C_b(\bar{D}\times [0, T])$ is a complete metric space, one can obtain by the Banach-Steinhaus theorem that
\[
\sup_n \|\rho^{(n)}\|<\infty.
\]
Then, one finds that
\[
|\int_0^T\int_{D} \varphi \rho(dx) dt|=\lim_{n\to\infty}|\int_0^T\int_{D} \varphi \rho^{(n)}(dx) dt|\le \liminf_{n\to\infty}\|\rho^{(n)}\|\|\varphi\|.
\]
This indicates that $\|\rho\|_D<+ \infty$.  
Since $C_b(\bar{D}\times [0, T])$ is separable, we can then find a countable dense set $\{\varphi_m\}\subset C_b$ such that for all $\varphi_m$ and almost every $t\in [0, T]$ such that
\[
\int_{D}\varphi_m(t)\rho_t^{(n)}(dx)\to  \int_{D}\varphi_m(t)\rho_t  (dx).
\]
This indicates that in fact for all $C_b(D)$ functions, this weak convergence holds for these $t$. Consequently, $\rho_t\in \mathcal{P}(D)$ for a.e. $t\in [0, T]$. Hence, we find find a version of $\rho$ such that $\rho_t\in \mathcal{P}(D)$ for all $t$.  Since $D$ is bounded, one then finds that
\begin{equation*}
    \int_0^1\int_{D} |x|^2 \rho(dx) dt < +\infty. 
\end{equation*}
Thus $\rho \in L^{1}([0,T];\mathcal{W}_{2}(D))$. 

For the space $Y$, it is straightforward by the Banach-Steinhaus theorem and the argument similarly above.  Hence, both $X$ and $Y$ are closed under weak topology.
\end{proof}

\subsection{Lower semi-continuity of the functionals}\label{sec:Lower semi-continuity}

First of all, we equip the product space $X \times Y$ for $(\rho,m)$ with the product topology and then the weak topology $(\rho_n,m_n) \Rightarrow (\rho,m)$ is understood as:  $\forall f \in C_{b}(\bar{D} \times [0,T];\mathbb{R})$, $g \in C_{b}(\bar{D} \times [0,T];\mathbb{R}^d)$, one has
\begin{equation}\label{def:topology of product space}
    \int_0^T\int_{D} f d\rho_{n}+\int_0^T\int_{D} g \cdot dm_{n} \rightarrow \int_0^T\int_{D} f d\rho+\int_0^T\int_{D} g \cdot dm.
\end{equation}

Next, we also rewrite the continuity equation in the weak sense as dual of $C_b^1$. In particular, we can define the subspace $\mathcal{H}$ of $X \times Y$ as following
\begin{equation*}
    \begin{aligned}
        \mathcal{H}:=\Bigg\{ (\rho,m) \in X \times Y: & -\int_0^T \int_{D}(\partial_t \varphi) \rho +\nabla \varphi \cdot m \ dxdt \\
        &+\int_{D} \varphi(x,T)\rho_{1}(x)-\varphi(x,0)\rho_{0}(x) dx =0, \forall \varphi \in C_{b}^{1}(\bar{D} \times [0,T]) \Bigg\}.
    \end{aligned}
\end{equation*}
Clearly $\mathcal{H}$ is closed due to linear constraints.

Corresponding to the optimization problems \eqref{minimization problem:OT indicator} and \eqref{minimization problem:OT-flow m rho alpha}, one naturally aims to introduce the functionals $F_{\alpha}$ and $F_{\infty}$ defined on $\mathcal{H}$ by:
\begin{equation*}
\begin{split}
&  F_{\infty}(\rho,m)=\int_0^T\int_{D} \frac{|m|^2}{\rho} dxdt+ \mathbf{1}_E,\\
& F_{\alpha}(\rho,m)=\int_0^T\int_{D} \frac{|m|^2}{\rho} dxdt+\alpha \mathbb{KL}\left[\rho(x,T)||\rho_1(x)\right].
\end{split}
\end{equation*}
Here we have the expressions like $|m|^2/\rho$ for measures $\rho$ and $m$, which must be defined. For this purpose, we recall the Benamou-Brenier functionals for $\rho \in \mathscr{M}(X)$ and $m \in \mathscr{M}^{d}(X)$:
\begin{equation*}
    \mathscr{B}_p(\rho, m):= \sup \left\{ \int_{X} a d\rho + \int_{X} b \cdot dm  : (a,b) \in C_b(X;K_q) \right\},
\end{equation*}
where $K_q:=\left\{(a, b) \in \mathbb{R} \times \mathbb{R}^d: a+\frac{1}{q}|b|^q \leq 0\right\}$ and $\frac{1}{p}+\frac{1}{q}=1$. It has the following characterizations.
\begin{lemma}\label{Thm:Benamou-Brenier formula}
    The functional  $\mathscr{B}_p$ is convex and lower semi-continuous on the space $\mathscr{M}(X) \times \mathscr{M}^{d}(X)$ for the weak convergence. Moreover, the following properties hold:
    \begin{itemize}
        \item  $\mathscr{B}_p(\rho,m) \geq 0$
        \item if both $\rho$ and $m$ are absolutely continuous with respect to a same positive measure $\lambda$ on $X$, we can write $\mathscr{B}_p(\rho,m)=\int_X f_{p}(\rho(x),m(x))d \lambda(x)$, where we identify $\rho(x)$ and $m(x)$ are the densities with respect to $\lambda$, and $f_p:\mathbb{R} \times \mathbb{R}^d \to \mathbb{R} \cup \{\infty\}$ is defined as:
        \begin{equation*}
        f_p(t, x):=\sup _{(a, b) \in K_q}(a t+b \cdot x)= \begin{cases}\frac{1}{p} \frac{|x|^p}{t^{p-1}} & \text { if } t>0 \\ 0 & \text { if } t=0, x=0 \\ +\infty & \text { if } t=0, x \neq 0, \text { or } t<0\end{cases}
        \end{equation*}
        \item $\mathscr{B}_p(\rho,m) < +\infty$ on if $\rho \geq 0$ and $m \ll \rho $ 
        \item for $\rho \geq 0$ and $m \ll \rho$, we have $m=v \cdot \rho$ and $\mathscr{B}_p(\rho,m)=\int \frac{1}{p} |v|^{p}d \rho$
    \end{itemize}
\end{lemma}
The detailed proof can be found in  \cite[Charp. 5, Proposition 5.18]{santambrogio2015optimal}. With the help of Benamou-Brenier functionals, we find that the the rigorous definitions of the functionals should be
\begin{equation*}
     F_{\infty}(\rho,m)=\int_0^T \mathscr{B}_2(\rho_t,m_t) dt+ \mathbf{1}_E,
     \quad \forall (\rho, m)\in \mathcal{H}.
\end{equation*}
\begin{equation*}
    F_{\alpha}(\rho,m)=\int_0^T \mathscr{B}_2(\rho_t,m_t) dt+\alpha \mathbb{KL}\left[\rho(x,T)||\rho_1(x)\right] \quad \forall (\rho, m)\in \mathcal{H}.
\end{equation*}
Then, the OT-flow problem \eqref{minimization problem:OT-flow m rho  1/alpha} is formulated by
\begin{equation}\label{eq:otflowrigor}
\min_{(\rho, m)\in \mathcal{H}}F_{\alpha}(\rho, m),
\end{equation}
and the OT problem \eqref{minimization problem:OT m rho} is given by:
\begin{equation}\label{eq:otrigor}
\min_{(\rho, m)\in \mathcal{H}}F_{\infty}(\rho, m).
\end{equation}

We note that the functionals just introduced enjoy good properties. 
\begin{proposition}
 Both $F_{\alpha}$ and $F_{\infty}$ are convex and lower semi-continuous. 
\end{proposition}
\begin{proof}
These two properties follow directly from those for the KL divergence and the 
Benamou-Brenier functional.

The convexity of these functionals are well-known. Here, we sketch the verification of the lower semi-continuity for the convenience of the readers (as the lower semi-continuity is essential in this work).  We take $F_{\alpha}$ as the example. 

If $\int_0^T \mathscr{B}_2(\rho_t,m_t)dt=+\infty$, then for each $M>0$, there is some $(\bar{a}, \bar{b})\in C_b(\bar{D}\times [0, T])$ such that
\[
\int_0^T \int_{D}  (\bar{a} \rho + \bar{b} \cdot m) \ dx dt>M.
\]
Consider any sequence $(\rho^n, m^n)\Rightarrow (\rho, m)$, by the weak convergence, one then has
\[
\lim_{n\to\infty}\int_0^T \int_{D}  (\bar{a} \rho^n(dx) + \bar{b} \cdot m^n(dx)) \, dt>M.
\]
This indicates that 
\[
\lim_{n\to\infty}\int_0^T \mathscr{B}_2(\rho_t^n,m_t^n)dt
\ge \lim_{n\to\infty}\int_0^T \int_{D}  (\bar{a} \rho^n(dx) + \bar{b} \cdot m^n(dx)) \, dt>M.
\]
Since $M$ is arbitrary, one then has $\lim_{n\to\infty}\int_0^T \mathscr{B}_2(\rho_t^n,m_t^n)dt=\infty$.

If $M:=\int_0^T \mathscr{B}_2(\rho_t,m_t)dt<+\infty$, then for any $\epsilon>0$, one may take 
$(\bar{a}, \bar{b})\in C_b(\bar{D}\times [0, T])$ such that
\[
\int_0^T \int_{D}  (\bar{a} \rho + \bar{b} \cdot m) \ dx dt<M+\epsilon.
\]
For any sequence $(\rho^n, m^n)\Rightarrow (\rho, m)$, by the weak convergence, one has
\begin{equation*}
    \begin{aligned}
            \int_0^T \mathscr{B}_2(\rho_t, m_t) dt  &< \int_0^T \int_{D}  (\bar{a} \rho(dx) + \bar{b} \cdot m(dx)) \ dt + \epsilon\\
            & \leq \liminf\limits_{n \to \infty}  \int_0^T \int_{D}  (\bar{a} \rho^{n}(dx) + \bar{b} \cdot m^{n}(dx)) \ dt + \epsilon\\
            & \leq \liminf\limits_{\alpha \to \infty}  \sup _{(a, b) \in C_b((D \times [0,T];K_2)}  \int_0^T \int_{D} (a \rho^{n}(dx) +b \cdot m^{n}(dx)) \ dt+ \epsilon \\
            & = \liminf\limits_{\alpha \to \infty} \int_0^T \mathscr{B}_2(\rho_t^n,m_t^n) dt + \epsilon.
    \end{aligned}
\end{equation*}
Since $\epsilon$ is arbitrary,  the lower semi-continuity of the first part in $F_{\alpha}$ is verified together.

The lower semi-continuity of KL divergence follows from a similar argument if one uses the weak form of the KL divergence (see appendix \ref{appendix:Lower semi-continuity of KL divergence}).    
\end{proof}

Though the space we consider is simply $L^1$ in time, We find that the time regularity of $\rho$ is actually good for feasible points.
\begin{proposition}
Fix $\alpha>0$ or $\alpha=\infty$. If $(\rho, m)$ is a feasible solution of the optimization problem \eqref{eq:otflowrigor} or \eqref{eq:otrigor}, then there is a version of $\rho$ such that $t\mapsto (\rho)_t$ is absolutely continuous in $W_2(D)$.  Moreover, $m \ll \rho$ and the Radon-Nikodym derivative
    \[
    v=\frac{dm}{d\rho}
    \]
    satisfies that $v\in L^1([0, T]; L^2(\rho_t))$.
\end{proposition}
\begin{proof}
    If $(\rho,m)$ is a feasible solution for either $F_{\alpha}$ or $F_{\infty}$, then 
    \begin{equation*}
        \int_0^T \mathscr{B}_2(\rho_t, m_t) dt < + \infty. 
    \end{equation*}
    By Lemma \ref{Thm:Benamou-Brenier formula},  for there is a set $E$ with Lebesgue measure zero such that  $t\in [0, T]\setminus E$, it holds that
    \begin{equation*}
        m_t \ll \rho_t,
    \end{equation*} 
    and we can construct the corresponding velocity field by 
    \begin{equation*}
        v_t=\frac{dm_t}{d\rho_t} . 
    \end{equation*}
Then, by Lemma \ref{Thm:Benamou-Brenier formula}, if we take $\lambda=\rho_t$, one then has for these $t$ that
\[
\mathscr{B}_2(\rho_t, m_t)=\int_{D}|v_t|^2 \rho_t(dx).
\]
Recall that $m\in L^1([0,1]; \mathscr{M}^d(D))$, then we can safely set $v_t=0$ for $t\in E$ and modify the value of $m$ on $E$ freely. Then, $m=\rho_t v_t$ holds for all $t\in [0, T]$. This means that $v_t\in L^1([0, T]; L^2(\rho_t))$
and the constraint $\partial_t\rho+\nabla\cdot m=0$ still holds for the modified version of $m$.
    
For the absolutely continuity of the feasible solution $\rho$ then follows from the constraint $\partial_t\rho+\nabla\cdot m=0$ in the weak sense (dual of $C_b^1$).
One can find the rigorous proof in \cite[Chap. 5, Theorem 5.14]{santambrogio2015optimal}. Here we provide a sketch of the proof. One can consider a segment of the curve, i.e. from $\rho_t$ to $\rho_{t+h}$. A natural transport plan from $\rho_t$ to $\rho_{t+h}$ can be given by the curve driven by $v_t$ :
    \begin{equation*}
        \gamma=(T_t,T_{t+h})_{\#}\rho_0,
    \end{equation*}
    where $T_t$ is defined by $\frac{d}{dt}T_t(x)=v_t(T_t(x))$. This plan provides an upper bound for $W_2\left(\rho_t, \rho_{t+h}\right)$:
    \begin{equation*}
        W_2\left(\rho_t, \rho_{t+h}\right) \leq\left(\int_{D \times D}|x-y|^2 d \gamma\right)^{1 / 2}=\left(\int_{D}\left|T_t(x)-T_{t+h}(x)\right|^2 d \rho_0\right)^{1 / 2}.
    \end{equation*}
    Using the fact that
    \begin{equation*}
        \left|T_t(x)-T_{t+h}(x)\right|^p \leq h \int_t^{t+h}\left|v_s\left(T_s(x)\right)\right|^2 d s ,
    \end{equation*}
   one has
    \begin{equation}\label{feasible solution:Lip rho}
        \frac{W_p\left(\rho_t, \rho_{t+h}\right)}{h} \leq\left(\frac{1}{h} \int_t^{t+h}\left\|v_s\right\|_{L^2\left(\rho_s\right)}^2 d s\right)^{1 / 2}.
    \end{equation}
    Since $\int_0^T \mathscr{B}_2(\rho,m) dt < + \infty$ automatically indicates 
    \begin{equation*}
        \int_0^T\|v\|_{L^2(\rho_t)}^2 dt < + \infty,
    \end{equation*}
    then \eqref{feasible solution:Lip rho} provides Lipschitz behavior for $\rho_t$, which directly leads to the fact $(\rho_t)_t$ is absolutely continuous in $\mathcal{W}_2(D)$.

\end{proof}


\subsection{$\Gamma$-convergence to optimal transport problems}

In this section, we aim to prove the $\Gamma$-convergence of OT-Flows to the optimal transport problems as $\alpha\to\infty$ so that the solutions could converge to the solution of the optimal transport problem.
We first show that the solutions (minimizers) indeed exist. 
\begin{proposition}\label{property of minizer}
Assume $D \subset \mathbb{R}^d$ is a bounded domain with smooth boundary, and $F_{\alpha}$ (resp. $F_{\infty}$) has at least a feasible point over $\mathcal{H}$. Then, \eqref{eq:otflowrigor} (resp. \eqref{eq:otrigor}) has global minimizers over $\mathcal{H}$. Moreover, the minimizers, as feasible points, satisfy that $\rho$ is absolutely continuous in $\mathcal{W}_2(D)$, and $m\ll \rho$.
\end{proposition}
\begin{proof}
First of all, note that $F_{\alpha}$ is nonnegative. Hence, it is clear that
\begin{equation*}
F^{\ast}_{\alpha}=\inf\limits_{(\rho,m)\in \mathcal{H}} F_{\alpha}(\rho,m) \ge 0.
\end{equation*}
With the assumption that one feasible point exist, $F^{\ast}_{\alpha}\in [0,\infty)$.

Consider a feasible minimizing sequence $(\rho_n,m_n)$ such that $F_{\alpha}(\rho_n,m_n) \to F_{\alpha}^{\ast}$. Let $\rho_n(x)$ and $m_n(x)$ be the densities with respect to some common nonnegative measure $\lambda^n_t$ (for example, one can take $\lambda_t^n=\rho_n$ and then $\rho_n(x)=1$). Then,   $\sup_n F_{\alpha}(\rho_n, m_n)<\infty$ implies that
\[
\sup\limits_{n}\int_0^T \int_{D}\frac{|m_n(x)|^2}{\rho_n(x)} \lambda_t^n(dx) dt< +\infty.
\]
One then has by the H\"older inequality that 
\[
\sup_n \|m_n\|=\int_0^T |m_n(x)|\lambda_t^n(dx) \leq \sup_n\left(\int_0^T \int_{D}\rho_n(x)\lambda^n_t(dx)dt  \int_0^T \int_{D}\frac{|m_n(x)|^2}{\rho_n(x)} \lambda_t^n(dx) dt\right)^{1/2}<\infty.
\]
  Hence, one has
    \begin{equation*}
        \sup\limits_{n} \|m_n\|+\|\rho_n\| < + \infty.
    \end{equation*}
    The Banach-Alaoglu thoeorem (the closed unit ball of the dual space of a normed vector space is compact in the weak star topology) tell us that there must be a weakly convergent subsequence. Together with the lower semi-continuity of $F_{\alpha}$, the minimizer of $F_\alpha$ exists.  Moreover, by proposition \ref{property of minizer}, we know the minimizer $(\rho_{\alpha},m_{\alpha})$ of functional $F_{\alpha}$ satisfies: $(\rho_{\alpha})_t$ is absolutely continuous in $W_2(D)$ and  $m_{\alpha} \ll \rho_{\alpha}$.
    
The argument for $F_{\infty}$ is similar. 
\end{proof}

\begin{theorem}
Assume $D \subset \mathbb{R}^d$ is a bounded domain with smooth boundary. The following results hold about the OT-flow and optimal transport problems.
\begin{enumerate}[(i)]
\item For any $\rho_0, \rho_1\in \mathcal{P}(D)$, $F_{\alpha}$ is $\Gamma$-convergent to $F_{\infty}$ as $\alpha\to \infty$. 

\item Suppose that both $F_{\alpha}$ and $F_{\infty}$ have feasible points. Let $(\rho^{\alpha},m^{\alpha})$ be an optimal solution to the corresponding minimization problem \eqref{eq:otflowrigor}. Then, for any increasing sequence $\{ \alpha_{I}\}$ going to infinity, where $I$ is an index set, there exists a convergent subsequence $(\rho^{\alpha_k},m^{\alpha_k}) \in \mathcal{H}$ with $\alpha_k \to + \infty$ such that the limit $(\rho^{\infty},m^{\infty}) $ is a solution of the minimization problem \eqref{eq:otrigor}.
\end{enumerate}
\end{theorem}

\begin{proof}

(i) 
We use $\Gamma_{\seq}$-convergence mentioned in Proposition \ref{connection} to prove the $\Gamma$-convergence from $F_{\alpha}$ to $F_{\infty}$. By definition, we need to verify: for any weakly convergent sequence $(\rho^{\alpha},m^{\alpha}) \Rightarrow (\rho,m)$, one have
    \begin{equation*}
        \inf\limits_{(\rho^{\alpha},m^{\alpha}) \to (\rho,m)} \limsup\limits_{\alpha \to \infty} F_{\alpha} (\rho^{\alpha},m^{\alpha}) \leq F_{\infty}(\rho,m) \leq \inf\limits_{(\rho^{\alpha},m^{\alpha}) \to (\rho,m)} \liminf\limits_{\alpha \to \infty} F_{\alpha} (\rho^{\alpha},m^{\alpha}).
    \end{equation*}
Above statement is equivalent to prove:
\begin{itemize}
    \item $\exists (\rho^{\alpha},m^{\alpha}) \Rightarrow (\rho,m)$, $ F_{\infty}(\rho,m) \geq \limsup\limits_{\alpha \to \infty} F_{\alpha} (\rho^{\alpha},m^{\alpha}). $
    \item $\forall (\rho^{\alpha},m^{\alpha}) \Rightarrow (\rho,m)$, $ F_{\infty}(\rho,m) \leq \liminf\limits_{\alpha \to \infty} F_{\alpha} (\rho^{\alpha},m^{\alpha}). $
\end{itemize}
For the first, one could take the constant sequence 
\[
(\rho^{\alpha},m^{\alpha})=(\rho, m).
\]
Then one has
\begin{equation*}
    \limsup\limits_{\alpha \to \infty} F_{\alpha}(\rho^{\alpha},m^{\alpha})=\limsup\limits_{\alpha \to \infty} F_{\alpha}(\rho,m) \leq F_{\infty} (\rho,m).
\end{equation*}
The above follows from the fact $\limsup\limits_{\alpha \to \infty} \alpha \mathbb{KL}\left[\rho(x,T)||\rho_1(x)\right]
\le \mathbf{1}_E$, which holds since  $\mathbb{KL}\left[\rho(x,T)||\rho_1(x)\right] \ge 0 $ and the equality holds only when $\rho(x,T)=\rho_{1}(x)$.

For the second, for any sequence $(\rho^{\alpha},m^{\alpha}) \Rightarrow (\rho,m)$, one needs to show that
\[
F_{\infty}(\rho, m)\le \lim_{\alpha\to\infty}
F_{\alpha}(\rho^{\alpha}, m^{\alpha}).
\]
We consider the two parts in  $F_{\alpha}$ and $F_{\infty}$ separately.
The required relation for the first part $\int_0^T \mathscr{B}_2(\rho_t, m_t)dt$ follows directly from the lower semi-continuity.

Considering the second part for the KL divergence, we can fix $n$. The lower semi-continuity of KL-divergence gives:
\begin{equation*}
    n \mathbb{KL} [\rho(x,T) || \rho_{1}(x)] \leq  n  \liminf\limits_{\alpha \to \infty} \mathbb{KL} [\rho^{\alpha}(x,T) || \rho_{1}(x)].
\end{equation*}
Let $n\to +\infty$, the left side becomes the indicator function we want:
\begin{equation*}
    \mathbf{1}_E\leq \lim\limits_{n \to \infty}  n  \liminf\limits_{\alpha \to \infty} \mathbb{KL} [\rho^{\alpha}(x,T) || \rho_{1}(x)] \leq \liminf\limits_{\alpha \to \infty}  \alpha\mathbb{KL} [\rho^{\alpha}(x,T) || \rho_{1}(x)].
\end{equation*}

Combining the two parts then yields that:
\begin{equation*}
     F_{\infty}(\rho,m) \leq \liminf\limits_{n \to \infty} F_{\alpha} (\rho^{\alpha},m^{\alpha}).
\end{equation*}
Thus the $\Gamma-$convergence from $F_{\alpha}$ to $F_{\infty}$ has then been established.

(ii)  Suppose $(\rho^{\ast},m^{\ast})$ is a feasible solution of problem \eqref{eq:otrigor} and thus also a feasible point of \eqref{eq:otflowrigor}. With the existence of feasible points, both $F_{\alpha}$ and $F_{\infty}$ have global minimizers. It is then clear that
\[
F_{\alpha}(\rho^{\alpha}, m^{\alpha})
\le F_{\alpha}(\rho^*, m^*)=F_{\infty}(\rho^*, m^*).
\]
This means that $F_{\alpha}(\rho^{\alpha}, m^{\alpha})$ is uniformly bounded in $\alpha$.
Again, by H\"older inequality as demonstrated  in Proposition \ref{property of minizer}, one has
\[
\sup\limits_{\alpha} \|m^{\alpha}\|+\|\rho^{\alpha}\| < + \infty.
\]

By the Banach-Alaoglu theorem, we conclude that any bounded set in $X \times Y$ is pre-compact. Consequently, there is a subsequence $(\rho^{\alpha_k},m^{\alpha_k}) \Rightarrow (\rho,m)$. Combining with $\Gamma-$convergence of the functionals, it follows that $(\rho,m)$ is a minimizer of $F_{\infty}$.
\end{proof}

For the existence of feasible points, one assumption is that both $\rho_0$ and $\rho_1$ are absolutely continuous with respect to the Lebesgue measure.
Then according to \eqref{eq:dynamic_ot}, there is feasible point $(\rho, m)$ with $\rho_T=\rho_1$ such that
\[
W_2(\rho_0, \rho_1)^2=\int_0^1 \mathscr{B}_2(\rho_t, m_t)dt<\infty.
\]
Hence, the feasible point exists. In the application, $\rho_1$ is a normal distribution, which is clearly absolutely continuous with respect to the Lebesgue measure. However, the data distribution $\rho_0$ is often singular (which may concentrate on some low-dimensional manifolds). The point is that one can always find a distribution $\tilde{\rho_0}$, which is absolutely continuous w.r.t Lebesgue measure to approximate $\rho_0$. In specific tasks, we only have samples from $\rho_0$ for optimization. The training process later can learn a velocity field that automatically generates the approximation $\tilde{\rho_0}$ that is absolutely continuous with respect to the Lebesgue measure \cite{jing2024machine}.

\section{Convergence of Monte Carlo approximation in the large data limit}\label{sec: convergence of N}

In practical applications, one only knows sample samples from $\rho_0$ to train the optimal velocity field. Hence, the the loss functional \eqref{eq:OT-Flow} in the OT flow is not known exactly, where the expectation is replaced by the Monte Carlo method (empirical mean over the samples).  Our goal in this part is to investigate the convergence of minimizers as the data size $N$ increases to infinity, provided sufficient approximation capability of neural networks. We will neglect the error brought by training process and assume that the minimization problem can be solved exactly for convenience. 

The research on the asymptotic behavior of minimizers in the large data limit is crucial in machine learning. It helps us gain a better understanding of the influence on the minimizers resulted from approximating loss functional by using data samples. As long as we are capable of providing an accurate estimation of the bound to control the error beyond the training sets, then we can guarantee some reliability on the generalisation of the neural network minimizers.

\subsection{Setup and the Monte Carlo approximation}

We first introduce some notations to define the finite dimensional spaces for the deep neural networks (see \cite{longo2021higher,loulakis2023new} for similar discussions).

A deep neural network $u : \mathbb{R}^d \to \mathbb{R}^{d^{\prime}}$ with $L$ layers is of the following form
\begin{equation}\label{network:2}
u(x; \theta):=F_L \circ \sigma \circ F_{L-1} \cdots \circ \sigma \circ F_1(x).
\end{equation}
The map $u(\cdot; \theta)$ is characterized by the hidden layers $F_k$, which are affine maps of the form
\begin{equation}\label{architecture of NN}
F_k y=W_k y+b_k, \quad \text { where } W_k \in \mathbb{R}^{d_{k+1} \times d_k}, b_k \in \mathbb{R}^{d_{k+1}}
\end{equation}
and $\sigma$ is the activation function which we assume to be Lipschitz continuous. We use 
$\theta$ to represent collectively all the parameters of the network $\mathcal{F}_L$, namely $W_k, b_k, k = 1,\cdots,L$. 
The network is trained using an optimizer such as stochastic gradient descent (SGD) \cite{lecun1998gradient} or ADAM \cite{DBLP:journals/corr/KingmaB14} commonly. In this paper, we will discuss the case of a multi-layer perceptron, but the results of convergence only rely on the approximation properties of the network, regardless of its precise architecture. 

We denote the set of all functions $u$ which can be given by networks with a given structure (fixed depth $L$ and width in each layer) of the form \eqref{network:2} as 
\begin{equation*}
V_{\mathcal{N}}=\left\{u: D \rightarrow \mathbb{R}^{d^{\prime}} \mid  \text{There exists a network \eqref{network:2} such that }u(x)=u(x;\theta)  \right\}.
\end{equation*}
Let $d_1$ be the dimension of $\theta$. Then, for each $\theta\in \R^{d_1}$, there is a corresponding $u\in V_{\mathcal{N}}$ such that
$u(x)=u(x;\theta)$.

\begin{lemma}\label{Convergence:uniform of F}
Assume $D$ is a bounded domain and the neural network space $V_{\mathcal{N}}$ is fixed. 
If the activation function $\sigma$ is assumed to be continuous, then given a sequence of network parameters $\{\theta_n\}_{n=1}^{\infty}$ that converges to $\theta$ in $\mathbb{R}^{d_1}$, $u(x;\theta_n)$ converges to $u(x;\theta)$ uniformly, or 
\[
\lim_{n\to\infty}\sup_{x\in D}|u(x; \theta_n)-u(x;\theta)|=0.
\]
\end{lemma}
\begin{proof}

By the specific form of the neural networks, it is clear that $u$ is continuous with respect to $(x, \theta)$.
If $\theta_n\to \theta$, then there exists $R>0$ such that
\[
\sup_n |\theta_n|\le R,\quad |\theta|\le R.
\]
Since the set $D\times B(0, R)$ is compact (bounded closed set in finitely dimensional space), then $u$ is uniformly continuous on $D\times B(0, R)$.

By the uniform continuity, it is easy to see that if $u(\cdot ;\theta_n)$ converges uniformly to $u(\cdot; \theta)$ for $D$ bounded.
\end{proof}

Let us come back to the optimization problem for the OT-flow problem \eqref{minimization problem:OT-flow v rho}.
To solve this, one seeks the optimal field $v$ using the neural networks (as mentioned, one may parametrize a scalar function $\Phi$ using the networks and set $v=-\nabla\Phi$). 
We will enforce the neural network for $v$ to satisfy
\begin{gather}\label{eq:bcvnn}
v(x, t)\cdot n=0, x\in \partial D.
\end{gather}
(This condition may be obtained by post-processing of the output of the neural network. ) Using the network for $v$, one may find the trajectory of a particle $t\mapsto z(x, t)$ using \eqref{eq:original ODE}.
Using the flow map (trajectory), our goal is then to solve the following problem
\[
\min_{v\in V_{\mathcal{N}}} J,
\]
where $J$ is given as in \eqref{eq:OT-Flow}.
The training of the algorithm assumes knowledge of data only at a finite set of points $\{x_i\}_{i=1}^{N}$. Then, the expectation is replaced by the Monte Carlo approximation
\[
J_N:=\frac{1}{N}\sum_i[C(x_i, T)+\frac{2}{\alpha}L(x_i, T)].
\]
This formulation is then used to train $v$ (find the optimal parameter $\theta\in \R^{d_1}$).

We are then concerned with the convergence as the number of data goes to infinity.
However, studying the convergence of $v$ is not convenient.  Instead, 
the loss in \eqref{eq:OT-Flow} naturally defines a function
\begin{equation}\label{loss function:F_v}
    F_v(x)=   \frac{1}{2}|\boldsymbol{z}(\boldsymbol{x}, T)|^{2}-\ell(x,T)+ \gamma_1\int_{0}^{T} \frac{1}{2}|v(z(x, t), t)|^{2} d t +\frac{d}{2} \log (2 \pi),
\end{equation}
for a given velocity field $v$, where $\ell$ and $z$ can be solved from ODE system \eqref{eq:OT-Flow}. We will then investigate the convergence of $F_v$ for the trained $v$.
For this purpose, consider the set of all functions $F_v$ with all possible $v$:
\begin{equation*}
    F_{\mathcal{N}} : =\{ F_v :v \in V_{\mathcal{N}} \}.
\end{equation*}
The OT-flow problem then corresponds to
\begin{equation}\label{continuous problem}
    \min _{F_v \in F_{\mathcal{N}}} \mathcal{E}(F_v):= \int_D F_v(x) \rho_0(dx),
\end{equation}
where $F_v$ denotes the loss functional for velocity $v$, $\rho_0$ is the data distribution.

The discrete loss functional can be defined as following:
\begin{equation}\label{discreteF}
  \min_{F_v\in F_{\mathcal{N}}}  \mathcal{E}_N(F_v)=\frac{1}{N} \sum_{i=1}^{N} F_v(x_i).
\end{equation}
Our goal is to investigate the convergence of the optimizers for \eqref{discreteF} to that for \eqref{continuous problem} as $N\to\infty$.

The unboundedness of the parameters space still brings trouble. In practice, parameter clipping is a commonly employed technique in neural networks aimed at preventing the function $u(x;\theta)$ from exploding.
Define the clipping parameter space
\begin{equation*}
    \Theta_{R} : = \{ \theta \in \mathbb{R}^{d_1}: |\theta| \leq R  \}.
\end{equation*}
With the clipped parameters, the function $F_v$ defined in \eqref{loss function:F_v} has good features as stated below. 
\begin{lemma}
\label{lemma:bounded F_v}
    Assume $D$ is a bounded domain. Fix the structure of the neural network space $V_{\mathcal{N}}$ with the activation function $\sigma$ to be twice continuously differentiable such that the condition \eqref{eq:bcvnn} is satisfied. Suppose the parameters $\theta \in \Theta_{R}$ for some fixed large $R>0$. 
    Then $v(x,t;\theta)$ is Lipschitz continuous and the Hessian $\nabla^2 v$ is bounded in the sense that
    \begin{equation*}
        \sup_{x\in D,\theta\in \Theta_R}|\frac{\partial^2 v}{\partial x_i \partial x_j}| <\infty, \forall i,j.
    \end{equation*}
    Consequently, $F_v$ is uniformly bounded and uniformly Lipschitz continuous with respect to $x$ (uniform for $\theta\in \Theta_R$ ) .
\end{lemma}
\begin{proof} 
Since the architecture of the neural networks is fixed as in \eqref{architecture of NN}, both $D$ and $\Theta_R$ are bounded, the Lipschitz continuity of $v(x,t;\theta)$ and the boundedness of the Hessian $\nabla^2v$ are straightforward since $\sigma$ is twice continuously differentiable, due to the fact that continuous functions on compact sets are bounded. 

For the velocity field $v$, since $v(x,t)$ is Lipschitz continuous and the boundary condition \eqref{eq:bcvnn} holds, then $z(x, T)$ is Lipschitz continuous with respect to $x$. Moreover, $z(x, T)$ is bounded.  In fact, by the classical ODE theory, there is one and only one solution curve passing through each point since $v$ is Lipschitz. 
We find that the solution curve passing through a point $x_0 \in \partial D$ stays on $\partial D$ due to the boundary condition.
More specifically, parameterize the boundary $\partial D$ by $\gamma(s)$ using the arc length parameter $s: [-\delta, \delta] \to \mathbb{R}^d$ for some $\delta>0$. Consider a curve of the force $x(t)=\gamma(s(t))$, then 
\[
x'(t)=\tau(s) s'(t)
\]
where $\tau$ is a unit tangent vector.
Since $v(x,t) \cdot n=0$, one finds that the equation 
\[
x'(t)=v(x(t), t) \Leftrightarrow 
s'(t)=v(\gamma(s(t)), t)\cdot \tau(s(t)). 
\]
This is a well-defined ODE for $s$. Solving this ODE gives a solution $t\mapsto s(t)$ and thus giving a curve $x(t)=\gamma(s(t))$ solving the ODE $\dot{x}=v(x(t), t)$. By the uniqueness of the solutions, there is no solution curve passing through $\partial D$ to the exterior of the domain. Hence, 
\[
z(x,t) \in D, \quad \forall t\in [0, T],
\]
which implies $z(x,T)$ is bounded. 
Hence, we conclude that $\frac{1}{2}|\boldsymbol{z}(\boldsymbol{x}, T)|^{2}$ is Lipschitz continuous with respect to $x$, and the Lipschitz constant is uniform for $\theta\in \Theta_R$.

We recall that $\ell$ satisfies the equation 
\[
\partial_t\ell(x,t)=\nabla\cdot v(x, t; \theta), \quad \ell(x,0)=0.
\]
Since $\nabla^2v$ is bounded,  $\operatorname{tr}(\nabla v)$ is Lipschitz continuous, which leads to the uniform Lipschitz continuity of $\ell(x,T)$.

The Lipschitz continuity of $\int_{0}^{T} \frac{1}{2}|v(z(x, t), t)|^{2} d t$ is a direct result of the Lipschitz continuity of $v$ and Lipscthitz continuity of $z(\cdot; t)$. The constant is uniform in $\theta$.

Hence, $F_v(\cdot; \theta)$ is uniformly Lipschitz continuous with respect to $x$.

Lastly, the uniform boundedness of $F_v$ is 
straightforward as we have shown that $z(x, t)$ is uniformly bounded and $v$ is continuous.
\end{proof}

With the clipped parameters and the help of 
Lemma \ref{lemma:bounded F_v},  we will the investigate the convergence of the optimizers for \eqref{discreteF} to that for \eqref{continuous problem} as $N\to\infty$ in the next section.


\subsection{Large data limit}

In this subsection, we explore the limit behavior of the minimizers as the sample size $N$ tends to $\infty$, i.e., provided sufficient data samples, for fixed structure of the neural networks and clipped parameters with $\theta\in \Theta_R$.
Although we neglect the effect of error brought by the approximation of the minimization problem (through, for example, stochastic gradient methods), the research on the behaviour of the minimizers as $N \to \infty$ is beneficial to understanding how machine learning algorithms work. If one can establish a bound to control the generalization error beyond the training sets, it offers theoretical stability and reliability for the generalization of the neural networks.

Suppose that we have a collection of data samples $\{X_i\}_{i=1}^N$ taking values in $D$, which are i.i.d. sampled from the data distribution $\rho_0$. 
Here, we would like to investigate the limit as $N\to\infty$. Since the samples are random samples, we put this into a probabilistic framework. In particular, by the Kolmogorov extension theorem \cite{tao2011introduction}, there is a probability space $(\Omega, \mathcal{F}, \mathbb{P})$ for the i.i.d. sampled $D$-valued random variables $\{X_i\}_{i=1}^{\infty}$ , with a common law $\rho_0\in \mathcal{P}(D)$.
The corresponding optimization problem \eqref{discreteF} becomes a probabilistic problem 
\begin{equation}\label{probabilistic problem}
    \min _{F_v \in F_{\mathcal{N}}} \mathcal{E}_{N,\omega}(F_v):=\frac{1}{N} \sum_{i=1}^N F_v(X_i(\omega))=\int_D F_v(x) \rho_{N,X(\omega)}(dx).
\end{equation}
Here,
\begin{equation}\label{eq:empiricalmeasure}
    \rho_{N}=\frac{1}{N}\sum_{i=1}^N \delta_{X_i(\omega)}.
\end{equation}
is the empirical measure and $\omega\in \Omega$ means one elementary event (a concrete realization of the $N$ samples). The following is obvious. 
\begin{lemma}
The problem \eqref{probabilistic problem} always has a solution $F_{N,\omega}\in V_{\mathcal{N}}$. 
\end{lemma}
\noindent This holds because the parameter space $\Theta_R$ is a compact set and the functional is continuous in $\theta$ when we fix $X_1,\cdots, X_N$.

Since the samples are i.i.d. sampled from $\rho_0$, one has
\begin{equation*}
    \mathbb{E}[\mathcal{E}_{N,\omega}(F_v)]=\int_D F_v(x)  \rho_0(dx).
\end{equation*}
Moreover, by the Law of Large numbers \cite{varadarajan1958convergence,tucker1959generalization}, the empirical measure \eqref{eq:empiricalmeasure} converges weakly to $\rho_0$. Here, we cite a stronger result for the convergence of empirical measures under $W_1$ distance. 
\begin{lemma}\label{rate estimate on W distance}
    Assume $D \subset \mathbb{R}^d$ is a bounded domain. Let $\mu \in \mathcal{P}(D)$.
    Consider an i.i.d. sequence $(X_k)_{k \geq 1}$ of $\mu$-distributed random variables and the empirical measure $\mu_N:= \frac{1}{N}\sum_{k=1}^N \delta_{X_{k}}$. Then there exist a constant $C$  such that for all $N \geq 1$,
    \begin{equation}
        \mathbb{E}\left[ W_1\left(\mu_N, \mu\right) \right] \leq C N^{-1/d}.
    \end{equation}
\end{lemma}
\noindent Note that $D$ is a bounded domain,  thus $\mu$ has $q$-th order moments for all $q$. Then this is a direct result of \cite[Theorem 1]{fournier2015rate} by setting $p=1$ and $q\gg 1$. 

Applying Lemma \ref{rate estimate on W distance} to our problem, we conclude that.
\begin{corollary}\label{lemma:W_1 convergence}
    Let $X_1, X_2, \cdots$ be a sequence of i.i.d. random variables with common law $\rho_0$. Consider the empirical measure in \eqref{eq:empiricalmeasure}.  Then 
    \[
    \lim_{N\to\infty}W_{1}( \rho_{N}, \rho_0) \to 0
    \]
    as $N \to +\infty$ almost surely.
\end{corollary}

With the preparation above, we now show for a.s. $\omega \in  \Omega$, the sequence of the minimizers $F_{N,\omega}$ has convergent subsequences and the limits would be a global minimizer of the continous OT-flow problem \eqref{continuous problem}. 
The tool we use here is again the $\Gamma$-convergence. Moreover, due to the simple structure of the fixed neural network, we can obtain much stronger results compared to previous section.

\begin{theorem}\label{thm:largedatalimit}
    Assume $D$ is a bounded domain with smooth boundary and fix the neural network space $V_{\mathcal{N}}$. Consider $F_v\in F_{\mathcal{N}}$ and $\theta \in \Theta_{R}$. Then the following holds. 
    \begin{enumerate}[(i)]
    \item
    Almost surely,  the functional $\mathcal{E}_{N,\omega}(F_v)$ in \eqref{probabilistic problem} is $\Gamma$-convergent to the functional $\mathcal{E}(F_v)$ in \eqref{continuous problem}, where the convergence of $F_v$ is the uniform convergence in $C_b(D)$.
    \item Consider the minimizer $F_{N,\omega}\in V_{\mathcal{N}}$ to \eqref{probabilistic problem}.
    Then, for almost surely $\omega \in \Omega$, any subsequence of the sequence $F_{N,\omega}$ has a convergent subsequence, with the limit $F_{\omega}^{\mathcal{N}}$ being a global minimizer to \eqref{continuous problem}, and it holds that
    \[
    \lim_{N\to\infty} \mathcal{E}_{N, \omega}\left(F_{N,\omega}\right)
    =\mathcal{E}\left(F_\omega^{\mathcal{N}}\right)=\inf _{\varphi \in V_{\mathcal{N}}} \mathcal{E}(\varphi).
    \]
    \end{enumerate}
\end{theorem}

We remark that the topology for $F_v\in F_{\mathcal{N}}$ is not important due to the correspondence to $\theta\in \Theta_R$, which is a subset of a finite dimensional space.

\begin{proof}[Proof of Theorem \ref{thm:largedatalimit}]
    (i) This is in fact straightforward. For any $F_N$ that converges uniformly to $F$, and since $\rho_{N}$ (defined in \eqref{eq:empiricalmeasure}) converges weakly converges to $\rho_0$, thus 
    \begin{equation}\label{eq:auxresult2}
        \lim_{N \rightarrow \infty} \int_D F_N \rho_{N}(dx)
        =\lim_{N \rightarrow \infty} \int_D (F_N-F) \rho_{N}(dx)
        +\lim_{N \rightarrow \infty}\int_D F \rho_N(dx)
        = \int_D F  \rho_0(dx)
    \end{equation}
    The first term here goes to zero because 
    \[
    |\int_D (F_N-F) \rho_{N}(dx)|\le \|F_N-F\|\to 0.
    \]
    The second term converges due to the weak convergence of $\rho_N$ almost surely by
    Corollary \ref{lemma:W_1 convergence}.
    
    (ii) First, for each $F_{N,\omega} \in V_{\mathcal{N}}$, we denote its corresponding neural network parameters in $\Theta_{R}$ as $\theta_{N,\omega}$. Since $| \theta_{N,\omega} | \leq R$, the sequence $\{ \theta_{N,\omega} \}$ has a subsequence that converges to some $\theta_{\omega}^{\mathcal{N}}$. We denote the corresponding neural network function of $\theta_{\omega}^{\mathcal{N}}$ as $F_{\omega}^{\mathcal{N}}$. By Lemma \ref{Convergence:uniform of F}, $F_{N,\omega}$ has a subsequence $\{ F_{N_k,\omega} \}$ that uniformly converges to $F_{\omega}^{\mathcal{N}}$. 

    By the $\Gamma$-convergence in Part (i), $F_{\omega}^{\mathcal{N}}$ is in fact a global minimizer of \eqref{continuous problem}:
    \begin{gather}\label{eq:result1}
    \mathcal{E}\left(F_\omega^{\mathcal{N}}\right)=\inf _{\varphi \in V_{\mathcal{N}}} \mathcal{E}(\varphi).
    \end{gather}
    
    To show the claim that the optimal value of the loss function also converges, our core idea is to prove the following inequalities for any convergent subsequence with limit $F_\omega^{\mathcal{N}} \in V_{\mathcal{N}}$:
    \begin{equation}\label{N:inf sup bound}
        \mathcal{E}\left(F_\omega^{\mathcal{N}}\right) \leq \liminf _{N \rightarrow \infty} \mathcal{E}_{N, \omega}\left(F_{N,\omega}\right) \leq \limsup _{N \rightarrow \infty} \mathcal{E}_{N, \omega}\left( F_{N,\omega} \right) \leq \inf _{\varphi \in V_{\mathcal{N}}} \mathcal{E}(\varphi).
    \end{equation}
    Due to \eqref{eq:result1}, the limit in fact exists along this subsequence. 
    Then, for any subsequence, there is a further subsequence such that the loss function value converges to $\inf _{\varphi \in V_{\mathcal{N}}} \mathcal{E}(\varphi)$, which is the same limit. Hence, the whole sequence converges.

    Next, we show \eqref{N:inf sup bound}. By the uniform convergence of the functions, it is clear (similar to \eqref{eq:auxresult2}) that
    \begin{equation*}
        \begin{aligned}
            \liminf _{N \rightarrow \infty} \mathcal{E}_{N, \omega}\left(F_{N,\omega}\right)= \mathcal{E}\left(F_\omega^{\mathcal{N}}\right).
        \end{aligned}
    \end{equation*}
    For the second inequality, the minimising property of $F_{N,\omega}$ implies that
    \begin{equation*}
        \mathcal{E}_{N, \omega}\left (F_{N,\omega} \right) \leq \mathcal{E}_{N, \omega}(F), \quad \text { for all } F \in V_{\mathcal{N}} .
    \end{equation*}
    Next, take a sequence $(F_n)_n$ in $V_{\mathcal{N}}$ that realises the inf in \eqref{N:inf sup bound}, i.e.,
    \begin{equation}\label{N:sup_proof_1}
        \lim _{n \rightarrow \infty} \mathcal{E}(F_n)=\inf _{\varphi \in V_{\mathcal{N}}} \mathcal{E}(\varphi).
    \end{equation}
    We recall the dual description of the $W_1$ distance \cite{santambrogio2015optimal}:
    \[
    W_1(\mu, \nu)=\sup_{\varphi: \|\varphi\|_{\mathrm{lip}}\le 1}\int \varphi\, d(\mu-\nu),
    \]
    where $\|\cdot\|_{\mathrm{lip}}$ means the least Lipschitz constant for the function.
    It then follows that
    \begin{equation*}
        \int_D F_n d(\rho_{N}-\rho_0) \leq C_L \cdot W_1(\rho_{N},\rho_0),
    \end{equation*}
    where $C_L$ is the upper bound of Lipschitz constant for all functionals in $F_{\mathcal{N}}$ indicated in Lemma \ref{lemma:bounded F_v}.
    By lemma \ref{lemma:W_1 convergence}, $W_1(\rho_{N},\rho_0) \to 0$ when  $N \to \infty$ almost surely. Thus
    \begin{equation*}
        \limsup _{N \rightarrow \infty} \mathcal{E}_{N, \omega}\left(F_{N,\omega} \right) \leq \limsup _{N \rightarrow \infty} \mathcal{E}_{N, \omega}\left(F_n\right) =  \mathcal{E}(F_n)
    \end{equation*} 
    holds for all $n$. Then we have
    \begin{equation*}
        \limsup _{N \rightarrow \infty} \mathcal{E}_{N, \omega}(F_{N,\omega}) \leq \inf_{n}  \mathcal{E}(F_n) = \inf _{\varphi \in V_{\mathcal{N}}} \mathcal{E}(\varphi).
    \end{equation*}
   The proof is thus complete.
\end{proof}

\subsection{Training the optimal velocity field}
To enhance the approximation ability of neural networks, we can  increase the complexity of neural networks and relax the clipping constant $R$ for parameters. Then the classical theory of universal approximation \cite{barron1993universal,cybenko1989approximation} ensures that we can select a sequence of neural network space $\{ V_{\ell}\}$ and  parameter space $\{ \Theta_{R_{\ell}} \}$ with increasing clipping constant $R_{\ell} \to \infty$, such that for each $v \in L^1([0, T]; L^2(\rho_t))$, there exists a $v_{\ell} \in V_{\ell}$ such that
\begin{equation*}
    \int_0^T \int_{D} |v-v_{\ell}|^2 \rho_t dxdt \to 0 \text{ \ as \ } \ell \to \infty,
\end{equation*}
which indicates that with increasing complexity of architecture, neural network will gradually have enough express capability to approximate the theoretical solutions of the OT problem. 

In real training, one can let $\ell \to \infty$ and $N \to \infty$ simultaneously. With delicately selected optimizers and hyperparameters, OT-Flow can find a minimizer close to the theoretical solution of problem \ref{minimization problem:OT-flow v rho}. Moreover, as $\alpha \to \infty$, the neural network minimizers of OT-Flow with different $\alpha$ will converges to classical OT problem solutions as we mentioned in section \ref{Convergence from OT-Flow to OT}.

\section{Conclusion and discussion}\label{sec:dis}
In summary, we have conducted two two primary convergence analyses for OT-Flow, one of the deep generative models. The first part employs $\Gamma$-convergence to establish the convergence of minimizers from OT-Flow to classical OT as the regularization coefficient $\alpha \to \infty $. The second part leverages the liminf-limsup framework to demonstrate the convergence of minimizers as the number of training samples $N \to \infty$. Furthermore, if we provide the neural network with sufficient approximation capability, the minimizers of OT-Flow will theoretically converges to classical OT ones. Our work enhances the understanding of the convergence properties of CNFs models with regularization, providing theoretical assurances for stability during training. Future research directions may involve applying similar methodologies to develop convergence analyses for other deep generative models, such as DPMs, and generative models associated with optimal control.

\subsection*{Acknowledgement}
This work is partially supported by the National Key R\&D Program of China No. 2020YFA0712000 and No. 2021YFA1002800. The work of L. Li was partially supported by NSFC 12371400 and 12031013, Shanghai Municipal Science and Technology Major Project 2021SHZDZX0102,  and Shanghai Science and Technology Commission (Grant No. 21JC1403700, 21JC1402900). The authors thank Ling Guo for helpful discussions and comments.

\appendix



\section{Lower semi-continuity of KL divergence}\label{appendix:Lower semi-continuity of KL divergence}
The central equation to derive the lower semi-continuity of KL divergence is the (Fenchel) dual formulation of the KL \cite[Lemma 9.4.4]{ambrosio2005gradient}:
\begin{equation}
    \mathbb{KL}[\mathbb{P}||\mathbb{Q}]= \sup\limits_{h \in C_{b}(\mathbb{R}^d)} \left\{  1+\int h d \mathbb{P} -\int e^{h} d \mathbb{Q} \right\}
\end{equation}
Thus KL divergence is lower semi-continuous with respect to $\mathbb{P}$ is a direct consequence of the fact that it is expressed as a supremum of linear functional. Suppose $\mathbb{P}^{n} \Rightarrow \mathbb{P}$, then  for all $\epsilon$, there exits some $\bar{h} \in C_{b}(\mathbb{R}^d) $ satisfies
\begin{equation}
    \begin{aligned}
             \mathbb{KL}[\mathbb{P}||\mathbb{Q}] &= \sup\limits_{h \in C_{b}(\mathbb{R}^d)} \left\{  1+\int h d \mathbb{P} -\int e^{h} d \mathbb{Q} \right\}\\
             & < 1+\int \bar{h} d \mathbb{P} -\int e^{ \bar{h}} d \mathbb{Q} + \epsilon \\
             & \leq \liminf\limits_{n \to \infty}  \left\{  1+\int \bar{h} d \mathbb{P}^{n} -\int e^{\bar{h}} d \mathbb{Q} \right\} + \epsilon\\
             & \leq \liminf\limits_{n \to \infty} \sup\limits_{h \in C_{b}(\mathbb{R}^d)} \left\{  1+\int h d \mathbb{P}^{n} -\int e^{h} d \mathbb{Q} \right\} + \epsilon
    \end{aligned}
\end{equation}
Thus we have $\mathbb{KL}[\mathbb{P}||\mathbb{Q}] \leq  \liminf\limits_{n \to \infty}  \mathbb{KL}[\mathbb{P}^{n}||\mathbb{Q}]$, i.e. lower semi-continuity.

\normalem
\bibliographystyle{plain}
\bibliography{gamma_OT}





\end{document}